\documentclass[english]{scrartcl}

\usepackage[T1]{fontenc}
\usepackage[utf8]{inputenc}

\usepackage{amsmath,amsfonts,amsthm,amssymb}
\usepackage[colorlinks,citecolor=blue]{hyperref}
\usepackage{doi,natbib}
\usepackage{cancel}

\usepackage{changes}

\usepackage[figure]{hypcap}
\usepackage{graphicx}
\usepackage{subcaption}

\usepackage{tikz}
\usetikzlibrary{arrows}

\tikzset{
    punkt/.style={
           rectangle,
           draw=white, very thick,
           text width=10.0em,
           minimum height=1.5em,
           text centered}
}

\usepackage{preamble}

\usepackage{marginnote}

\usepackage[capitalise,nameinlink]{cleveref}
\allowdisplaybreaks

\newcommand\norm[2]{\left\Vert#1\right\Vert_{#2}}

\newcommand\N{\mathbb{N}}
\newcommand\R{\mathbb{R}}

\newcommand{\dist}{\operatorname{dist}}
\newcommand{\dom}{\operatorname{dom}}
\newcommand{\gph}{\operatorname{gph}}
\newcommand{\intr}{\operatorname{int}}
\DeclareMathOperator*{\argmin}{\operatorname{argmin}}

\newtheorem{theorem}{Theorem}[section]
\newtheorem{lemma}[theorem]{Lemma}

\newtheorem{remark}[theorem]{Remark}
\newtheorem{definition}[theorem]{Definition}
\newtheorem{example}[theorem]{Example}

\crefname{figure}{Figure}{Figures}

\begin{document}

\title{A note on partial calmness for bilevel optimization problems with linearly structured lower level}
\author{%
	Patrick Mehlitz%
	\footnote{%
		Brandenburgische Technische Universit\"at Cottbus--Senftenberg,
		Institute of Mathematics,
		03046 Cottbus,
		Germany,
		\email{mehlitz@b-tu.de},
		\url{https://www.b-tu.de/fg-optimale-steuerung/team/dr-patrick-mehlitz},
		ORCID: 0000-0002-9355-850X%
		}
	\and
	Leonid I.\ Minchenko%
	\footnote{%
		Belarus State University of Informatics and Radioelectronics,
		6 P.\ Brovki Street,
		Minsk 220013,
		Belarus,
		\email{leonidm@insoftgroup.com},
		ORCID: 0000-0002-8773-2559
		}
	\and
	Alain B.\ Zemkoho%
	\footnote{%
		University of Southampton,
		School of Mathematics,
		Southampton SO17 1BJ,
		United Kingdom,
		\email{a.b.zemkoho@soton.ac.uk},
		\url{https://www.southampton.ac.uk/maths/about/staff/abz1e14.page},
		ORCID: 0000-0003-1265-4178
		}
	}

\publishers{}
\maketitle

\begin{abstract}
	Partial calmness is a celebrated but restrictive property of
	bilevel optimization problems whose presence opens a way to
	the derivation of Karush--Kuhn--Tucker-type necessary optimality conditions in
	order to characterize local minimizers.
	In the past, sufficient conditions for the validity of
	partial calmness have been investigated. In this regard,
	the presence of a linearly structured lower level problem
	has turned out to be beneficial. However, the associated
	literature suffers from inaccurate results.
	In this note, we clarify some regarding erroneous statements
	and visualize the underlying issues with the aid of 
	illustrative counterexamples.
\end{abstract}


\begin{msc}	
	49J53, 90C30
\end{msc}

\section{Introduction}\label{sec:introduction}

We consider the standard bilevel optimization problem
\begin{equation}\label{eq:BPP}\tag{BPP}
	\min\limits_{x,y}\{F(x,y)\,|\,x\in X,\,y\in S(x)\},
\end{equation}
where $F\colon\R^n\times\R^m\to\R$ is a locally Lipschitz continuous function, the set 
$X\subset\R^n$ is nonempty and closed, and $S\colon\R^n\rightrightarrows\R^m$
is the solution mapping of a standard parametric optimization problem, i.e.,
\[
	\forall x\in\R^n\colon\quad
		S(x):=\argmin_y\{f(x,y)\,|\,g(x,y)\leq 0\}.
\]
Above, the data functions $f\colon\R^n\times\R^m\to\R$ and $g\colon\R^n\times\R^m\to\R^q$
are assumed to be continuous. The component functions of $g$ will be addressed
by $g_1,\ldots,g_q\colon\R^n\times\R^m\to\R$.
Nowadays, bilevel programming is one of the most intensively investigated topics in optimization theory
since, on the one hand, there exist numerous underlying applications from economics, 
finance, chemistry, or engineering
while, on the other hand, problems of this type are quite challenging, 
both theoretically and numerically, see \cite{Bard1998,Dempe2002,DempeKalashnikovPerezValdesKalashnykova2015}.
In this note, we are concerned with the concept of partial calmness, fundamental in deriving 
necessary optimality conditions when the so-called optimal value function reformulation 
(a precise definition is stated below) is under consideration, see
\cref{sec:partial_calmness} for details.

To introduce this reformulation, we exploit the function $\varphi\colon\R^n\to\overline{\R}$ given by
\[
	\forall x\in\R^n\colon\quad\varphi(x):=\inf\limits_y\{f(x,y)\,|\,g(x,y)\leq 0\}.
\]
Clearly, $\varphi$ is the so-called optimal value function of the parametric
optimization problem $\min_y\{f(x,y)\,|\,g(x,y)\leq 0\}$ which is referred
to as the lower level problem of \eqref{eq:BPP}.
It is well known that \eqref{eq:BPP} is equivalent to its 
optimal value function reformulation defined by
\begin{equation}\label{eq:OVR}\tag{OVR}
	\min\limits_{x,y}\{F(x,y)\,|\,x\in X,\,f(x,y)-\varphi(x)\leq 0,\,g(x,y)\leq 0\}.
\end{equation}
Observing that \eqref{eq:OVR} is a single-level optimization problem, this reformulation
approach opens a way to the theoretical and numerical treatment of \eqref{eq:BPP}.
Due to the implicit character of $\varphi$, one has to observe that \eqref{eq:OVR}
is still a quite challenging problem since it is often not possible to
compute a fully explicit representation of the function $\varphi$ in practice.
Additionally, \eqref{eq:OVR} is generally nonsmooth
since $\varphi$ is nonsmooth in several practically relevant situations.
Moreover, it is folklore that \eqref{eq:OVR} is inherently irregular by definition of
$\varphi$, i.e., standard constraint qualifications from (nonsmooth) optimization fail to hold at
all feasible points of \eqref{eq:OVR}, see e.g.\
\cite[Proposition~3.2]{YeZhu1995}.
Finally, let us mention that \eqref{eq:OVR} is a nonconvex optimization problem in
general even if all the data functions $F$, $f$, and $g_1,\ldots,g_q$ as well as the
set $X$ are convex.
Despite all these shortcomings, reformulating
\eqref{eq:BPP} via \eqref{eq:OVR} became quite popular in the mathematical programming
community. Starting with \cite{Outrata1988}, there appeared numerous publications which
exploit \eqref{eq:OVR} for the derivation of optimality conditions and solution
algorithms for \eqref{eq:BPP}.

In the seminal paper \cite{YeZhu1995}, the authors suggested to investigate the situation
where $(x,y)\mapsto f(x,y)-\varphi(x)$ is a locally exact penalty function for \eqref{eq:OVR}
at some of its local minimizers $(\bar x,\bar y)\in\R^n\times\R^m$ in more detail. 
More precisely, they discussed conditions ensuring the existence of a finite scalar
$\bar\kappa>0$ such that $(\bar x,\bar y)$ is a local minimizer of the partially
penalized problem 
\begin{equation}\label{eq:partially_penalized_OVR}\tag{OVR$(\kappa)$}
	\min\limits_{x,y}\{F(x,y)+\kappa(f(x,y)-\varphi(x))\,|\,x\in X,\,g(x,y)\leq 0\}
\end{equation}
for all $\kappa\geq\bar\kappa$, too. 
The authors called this property \emph{partial calmness} of \eqref{eq:BPP} at
$(\bar x,\bar y)$, see \cref{def:partial_calmness} and 
\cref{lem:partial_calmness_and_exact_penalization}, according to Clarke's classical
notion of calmness for nonlinear optimization problems, see \cite[Section~6.4]{Clarke1983}.
Observing that \eqref{eq:partially_penalized_OVR}
may satisfy standard constraint qualifications, the presence of the 
partial calmness property opens a way to the derivation of
necessary optimality conditions for \eqref{eq:BPP} via \eqref{eq:OVR}.
This has been done successfully in e.g.\
\cite{DempeDuttaMordukhovich2007,DempeFranke2015,DempeZemkoho2011,DempeZemkoho2012,DempeZemkoho2013,
MordukhovichNamPhan2012,YeZhu1995,YeZhu2010}.
Recently, some Newton-type methods for the numerical solution of
\eqref{eq:BPP} have been developed which are based on the presence of the
partial calmness property, see
\cite{FischerZemkohoZhou2019,FliegeTinZemkoho2020}.
Let us note that the idea of partial calmness can be generalized to far more difficult
settings, e.g.\ to bilevel optimal control problems or to situations where the lower level program
is a parametric conic optimization problem, see e.g.\
\cite{BenitaMehlitz2016,DempeMefoMehlitz2018,Mehlitz2016,Ye1995,Ye1997}.
Unfortunately, partial calmness is a quite restrictive property, see
\cite{HenrionSurowiec2011}, which only holds in very particular
situations, e.g.\ where the lower level problem of \eqref{eq:BPP} is fully linear, i.e., when
the functions $f$ and $g$ are affine w.r.t.\ all variables, see
\cite[Proposition~4.1]{YeZhu1995}. The latter result gave rise to a number of publications where
the authors tried to generalize this observation to lower level problems where linearity
is only present w.r.t.\ $y$, see e.g.\ \cite{DempeZemkoho2012,DempeZemkoho2013}. However, as we will see
in this note, such a generalization is not possible in general.
We present simple counterexamples which refute more general versions of \cite[Proposition~4.1]{YeZhu1995}.
Furthermore, we point out the essential bug, originating from \cite{YeZhu1995},
which caused the proof in e.g.\ \cite[Theorem~4.2]{DempeZemkoho2013} to be erroneous.

The remaining parts of this note are organized as follows.
In \cref{sec:notation}, we briefly summarize the notation and terminology used in this manuscript.
We formally introduce the concept of partial calmness and recall some sufficient conditions
guaranteeing its validity in \cref{sec:partial_calmness}.
\Cref{sec:linear_lower_level} is dedicated to the investigation of partial calmness in the
context of bilevel optimization with linearly structured lower level problems.
We first state a correct proof of the seminal result \cite[Proposition~4.1]{YeZhu1995}
which addresses fully linear lower level problems. Furthermore, we comment on the
bug from the classical proof stated in \cite{YeZhu1995}. 
By means of examples, we visualize that partial calmness
does not need to be inherent as soon as the lower level problem is only linear w.r.t.\
the variable $y$. Reviewing some literature, we report on selected conditions which ensure
validity of partial calmness in this situation.
We finalize the paper by means of some concluding remarks in \cref{sec:conclusion}.

\section{Notation}\label{sec:notation}

In this manuscript, we mainly exploit standard notation.
Without loss of generality, we equip all appearing spaces (including product structures)
with the maximum norm $\norm{\cdot}{\infty}$.
For some vector $x\in\R^n$ and a scalar $\varepsilon>0$,
$\mathbb U_\varepsilon(x):=\{y\in\R^n\,|\,\norm{x-y}{\infty}<\varepsilon\}$
and $\mathbb B_\varepsilon(x):=\{y\in\R^n\,|\,\norm{x-y}{\infty}\leq\varepsilon\}$
denote the open and closed $\varepsilon$-ball around $x$, respectively.
Furthermore, for an arbitrary set $A\subset\R^n$, we use
\[
	\dist(x,A):=\inf\{\norm{x-y}{\infty}\,|\,y\in A\}
\]
in order to represent the distance of $x$ to $A$.
For some matrix $M\in\R^{m\times n}$ and an index set $I\subset\{1,\ldots,m\}$,
$M_I\in\R^{|I|\times n}$ is the matrix which results from $M$ by deleting all
rows whose associated index does not belong to $I$.
We use $\mathtt e\in\R^n$ to denote the all-ones vector.

Let $\Upsilon\colon\R^n\rightrightarrows\R^m$ be a set-valued mapping.
The domain and the graph of $\Upsilon$ are
defined by $\dom\Upsilon:=\{x\in\R^n\,|\,\Upsilon(x)\neq\varnothing\}$
and $\gph\Upsilon:=\{(x,y)\in\R^n\times\R^m\,|\,y\in\Upsilon(x)\}$, respectively.
Recall that $\Upsilon$ is called inner semicontinuous at some point
$(\bar x,\bar y)\in\gph\Upsilon$ 
whenever for each sequence $\{x_k\}_{k\in\N}\subset\R^n$
converging to $\bar x$, there is another sequence $\{y_k\}_{k\in\N}\subset\R^m$ converging to
$\bar y$ such that $y_k\in\Upsilon(x_k)$ holds for all sufficiently large $k\in\N$.
Assume that there are continuous functions $h_1,\ldots,h_q\colon\R^n\times\R^m\to\R$
such that $\Upsilon$ possesses the particular form
\[
	\forall x\in\R^n\colon\quad
	\Upsilon(x):=\{y\in\R^m\,|\,h_i(x,y)\leq 0,\,i=1,\ldots,q\}
\]
and that $(\bar x,\bar y)\in\gph\Upsilon$ is chosen arbitrarily.
Then $\Upsilon$ is called R-regular at $(\bar x,\bar y)$ w.r.t.\ $\Omega\subset\R^n$
whenever there are constants $\kappa>0$ and $\varepsilon>0$ such that
\begin{align*}
	&\forall (x,y)\in\mathbb U_\varepsilon(\bar x,\bar y)\cap(\Omega\times\R^m)\colon\\
	&\qquad\dist(y,\Upsilon(x))\leq \kappa\,\max\{0,\max\{h_i(x,y)\,|\,i\in\{1,\ldots,q\}\}\}
\end{align*}
holds, see e.g.\ \cite{BednarczukMinchenkoRutkowski2019}.
In case where this condition holds for $\Omega:=\R^n$, 
we simply say that $\Upsilon$ is R-regular at $(\bar x,\bar y)$.
Typically, one of the settings $\Omega:=\R^n$ or $\Omega:=\dom\Upsilon$ is under consideration.
Due to \cite[Theorem~5.1]{BednarczukMinchenkoRutkowski2019}, R-regularity of $\Upsilon$ at some point 
of its graph is stronger than the so-called Aubin property, which is a prominent Lipschitzian property of
set-valued mappings, provided that the functions $h_1,\ldots,h_q$ are locally
Lipschitz continuous in a neighbourhood of the point of interest.

\section{Bilevel optimization and partial calmness}\label{sec:partial_calmness}

Let us state the classical definition of partial calmness due to \cite{YeZhu1995}.
This property demands the problem \eqref{eq:OVR} to behave stably in a certain sense
w.r.t.\ small perturbations of the constraint $f(x,y)-\varphi(x)\leq 0$ at a given
local minimizer of \eqref{eq:BPP}. Partial calmness originates from the classical notion
of calmness for standard nonlinear optimization problems introduced in \cite[Section~6.4]{Clarke1983}.
\begin{definition}\label{def:partial_calmness}
	Let $(\bar x,\bar y)\in\R^n\times\R^m$ be a local minimizer of \eqref{eq:BPP}.
	We say that \eqref{eq:BPP} is \emph{partially calm} at $(\bar x,\bar y)$ if
	there exist $\delta>0$ and $\bar\kappa>0$ such that for each triplet
	$(x,y,u)\in\mathbb B_\delta(\bar x,\bar y,0)$ satisfying
	\[
		x\in X,\quad f(x,y)-\varphi(x)\leq u,\quad g(x,y)\leq 0,
	\]
	we have $F(x,y)+\bar\kappa|u|\geq F(\bar x,\bar y)$.
\end{definition}

As we already mentioned in \cref{sec:introduction}, it is clear from \cite[Proposition~3.3]{YeZhu1995}
that partial calmness of \eqref{eq:BPP} at one of its local minimizers $(\bar x,\bar y)$ is equivalent to
$(x,y)\mapsto f(x,y)-\varphi(x)$ being a locally exact penalty function for \eqref{eq:OVR} at $(\bar x,\bar y)$. We summarize this observation in the subsequently stated lemma.
\begin{lemma}\label{lem:partial_calmness_and_exact_penalization}
	Let $(\bar x,\bar y)\in\R^n\times\R^m$ be a local minimizer of \eqref{eq:BPP}.
	Then \eqref{eq:BPP} is partially calm at $(\bar x,\bar y)$ if and only if there is some $\bar\kappa>0$
	such that $(\bar x,\bar y)$ is a local minimizer of \eqref{eq:partially_penalized_OVR} for
	each $\kappa\geq\bar\kappa$.
\end{lemma}

Next, we want to mention some criteria ensuring validity of the partial calmness
condition.
Therefore, we first introduce two set-valued mappings $\Gamma,\Phi\colon\R^n\rightrightarrows\R^m$
as stated below:
\[
	\begin{aligned}
		&\forall x\in\R^n\colon\quad&
		\Gamma(x)&:=\{y\in\R^m\,|\,g(x,y)\leq 0\},&\\
		&&
		\Phi(x)&:=\{y\in\R^m\,|\,f(x,y)-\varphi(x)\leq 0,\,g(x,y)\leq 0\}.
	\end{aligned}
\]
Let us mention that $\gph S$ and $\gph\Phi$ actually coincide by definition of the
optimal value function $\varphi$. However, $\Phi$ characterizes the lower level
solution set with the aid of standard inequality constraints comprising the implicitly
given function $\varphi$ while $S$ is completely implicit.

The following definitions are motivated by the considerations in the papers
\cite{BednarczukMinchenkoRutkowski2019,HenrionSurowiec2011,YeZhu1995}, which,
at least in parts, are concerned with bilevel optimization.
\begin{definition}\label{def:partial_calmness_CQs}
Fix some point $(\bar x,\bar y)\in\gph S$.
\begin{itemize}
  \item[(i)] We say that the lower level problem satisfies the
  		\emph{local uniformly weak sharp minimum condition} (LUWSMC)
  		at $(\bar x,\bar y)$ whenever there are some $\alpha>0$ and $\varepsilon>0$
  		such that the following condition holds:
  		\begin{equation}\label{eq:LUWSMC}
  			\forall (x,y)\in \gph\Gamma\cap\mathbb U_\varepsilon(\bar x,\bar y)\colon
  			\quad
  			\alpha\,\dist(y,S(x))\leq f(x,y)-\varphi(x).
  		\end{equation}
  		If the above condition can be strengthened to
  		\[
  			\forall (x,y)\in \gph\Gamma\colon
  			\quad \alpha\,\dist(y,S(x))\leq f(x,y)-\varphi(x),
  		\]
  		then the lower level problem is said to have a \emph{uniformly weak sharp minimum} (UWSM).
  		Particularly, the latter condition is independent of 
  		the point of interest $(\bar x,\bar y)$.
  \item[(ii)] We say that the \emph{R-regularity constraint qualification} (RRCQ) holds at
  		$(\bar x,\bar y)$ whenever $\Phi$ is R-regular at $(\bar x,\bar y)$ w.r.t.\ $\dom\Phi$.
\end{itemize}
\end{definition}

By definition, validity of UWSM implies that LUWSMC holds at all points in $\gph S$.
Both conditions originate from the notion of weak sharp minima which addresses constrained 
optimization problems, see \cite{BurkeFerris1993}.
Note that validity of RRCQ at some point $(\bar x,\bar y)\in\gph S$ implicitly demands that $\varphi$
is continuous around $\bar x$.

Below, we study the relationship between the conditions LUWSMC and RRCQ
as well as their connection to the partial calmness property.
First, we would like to show that whenever a given local minimizer of \eqref{eq:BPP} satisfies
LUWSMC, then \eqref{eq:BPP} is partially calm at this point. 
This result generalizes related observations
from \cite[Proposition~5.1]{YeZhu1995} or \cite[Propositions~3.8 and~3.10]{HenrionSurowiec2011}.
\begin{lemma}\label{lem:LUWSM_ensures_VFCQ}
	Let $(\bar x,\bar y)\in\R^n\times\R^m$ be a local minimizer of
	\eqref{eq:BPP} where LUWSMC holds.
	Then \eqref{eq:BPP} is partially calm at $(\bar x,\bar y)$.
\end{lemma}
\begin{proof}
	By local optimality of $(\bar x,\bar y)$ for \eqref{eq:BPP}, we find a constant
	$\gamma>0$ such that $F(x,y)\geq F(\bar x,\bar y)$ holds for all 
	$(x,y)\in\mathbb B_\gamma(\bar x,\bar y)$ which are feasible to \eqref{eq:BPP}.
	Without loss of generality, we may assume that $F$ is Lipschitz continuous on
	$\mathbb B_\gamma(\bar x,\bar y)$ with Lipschitz constant $L>0$. 
	The assumptions of the lemma guarantee the existence of constants $\alpha>0$ 
	and $\varepsilon>0$ such that \eqref{eq:LUWSMC} holds.
	Set $\delta:=\min\{\varepsilon/2,\max\{\alpha,1\}\gamma/2\}$ and fix 
	a triplet $(x,y,u)\in\mathbb B_\delta(\bar x,\bar y,0)$ satisfying 
	$x\in X$, $f(x,y)-\varphi(x)\leq u$,
	and $g(x,y)\leq 0$. Due to $(x,y)\in\gph\Gamma\cap\mathbb U_\varepsilon(\bar x,\bar y)$,
	we find some point $\tilde y\in S(x)$ such that the estimate
	$\norm{y-\tilde y}{\infty}\leq(1/\alpha)(f(x,y)-\varphi(x))$ holds. This yields
	\begin{align*}
		\norm{\tilde y-\bar y}{\infty}
		\leq
		\norm{\tilde y-y}{\infty}+\norm{y-\bar y}{\infty}
		\leq
		(1/\alpha)u+\norm{y-\bar y}{\infty}
		\leq
		\delta/\alpha+\delta
		\leq
		\gamma,
	\end{align*}
	and due to $x\in\mathbb B_\gamma(\bar x)$, we find $(x,\tilde y)\in\mathbb B_\gamma(\bar x,\bar y)$.
	Furthermore, this point is feasible to \eqref{eq:BPP} which is why we obtain
	\begin{align*}
		F(x,y)-F(\bar x,\bar y)
		&
		\geq
		F(x,y)-F(x,\tilde y)
		\geq
		-L\norm{y-\tilde y}{\infty}\\
		&
		\geq
		-(L/\alpha)(f(x,y)-\varphi(x))
		\geq
		-(L/\alpha)u.
	\end{align*}
	Rearranging this inequality while noting that $(x,y,u)\in\mathbb B_\delta(\bar x,\bar y,0)$
	has been chosen arbitrarily, the lemma's assertion follows.
\end{proof}

Next, we show that RRCQ is a sufficient condition for LUWSMC.
This generalizes the recent result \cite[Theorem~6.1]{BednarczukMinchenkoRutkowski2019}.
\begin{lemma}\label{lem:RRCQ_implies_LUWSMC}
	Let $(\bar x,\bar y)\in\gph S$ be chosen where RRCQ holds.
	Furthermore, assume that there is some neighbourhood $U\subset\R^n$
	of $\bar x$ such that $(\dom\Gamma)\cap U=(\dom S)\cap U$ is valid.
	Then LUWSMC holds at $(\bar x,\bar y)$.
\end{lemma}
\begin{proof}
	Noting that RRCQ holds at $(\bar x,\bar y)$, the mapping $\Phi$ is R-regular
	at $(\bar x,\bar y)$ w.r.t.\ $\dom\Phi$.
	Consequently, we find $\kappa>0$ and some $\varepsilon>0$ such that
	\begin{align*}
		&\forall (x,y)\in\mathbb U_\varepsilon(\bar x,\bar y)\cap(\dom\Phi\times\R^m)\colon\\
		&\qquad \dist(y,\Phi(x))\leq \kappa\max\{0,f(x,y)-\varphi(x),g_1(x,y),\ldots,g_q(x,y)\}
	\end{align*}
	holds. For each pair $(x,y)\in\mathbb U_\varepsilon(\bar x,\bar y)$
	with $g(x,y)\leq 0$, we automatically have the inequality $f(x,y)\geq\varphi(x)$ by definition
	of $\varphi$. Thus, we obtain
	\begin{align*}
		&\forall (x,y)\in\mathbb U_\varepsilon(\bar x,\bar y)\cap(\dom\Phi\times\R^m)\colon\\
		&\qquad (x,y)\in\gph\Gamma\,\Longrightarrow\,\dist(y,\Phi(x))\leq\kappa(f(x,y)-\varphi(x)).
	\end{align*}
	Due to $\Phi(x)=S(x)$ for all $x\in\R^n$ and $(\dom\Phi)\cap U=(\dom\Gamma)\cap U$, this implies
	\begin{align*}
		&\forall (x,y)\in\mathbb U_\varepsilon(\bar x,\bar y)\cap(U\times\R^m)\colon\\
		&\qquad (x,y)\in\gph\Gamma\,\Longrightarrow\,\dist(y,S(x))\leq\kappa(f(x,y)-\varphi(x)).		
	\end{align*}
	Observing that we can find an open ball around $(\bar x,\bar y)$ which is contained
	in the intersection $\mathbb U_\varepsilon(\bar x,\bar y)\cap(U\times\R^m)$, division by $\kappa$
	shows that LUWSMC is valid at $(\bar x,\bar y)$.
\end{proof}

Note that the existence of a neighbourhood $U\subset\R^n$ of some given point 
$\bar x\in\dom\Gamma$
such that $(\dom \Gamma)\cap U=(\dom S)\cap U$ holds is not too restrictive. By continuity of
$g$, we already know that $\Gamma$ possesses closed images. Thus, if $\Gamma$ is locally
bounded around $\bar x$, the above condition is inherent by Weierstra\ss' theorem.

Recently, some conditions, which comprise a local constant rank assumption,
implying validity of R-regularity have been derived in
\cite{BednarczukMinchenkoRutkowski2019}.
Applying this to the situation at hand, we obtain the following result.
\begin{lemma}\label{lem:RRCQ_via_CRCQ_and_lower_semicontinuity}
	Fix some point $(\bar x,\bar y)\in\gph S$ where $S$ is inner semicontinuous.
	Furthermore, let the functions $f$ and $g_1,\ldots,g_q$ be locally Lipschitz continuous
	and continuously differentiable
	w.r.t.\ $y$ in a neighbourhood of $(\bar x,\bar y)$.
	Let us denote 
	the index set associated with all lower level inequality constraints which are
	active at $(\bar x,\bar y)$ by
	$I(\bar x,\bar y):=\{i\in\{1,\ldots,q\}\,|\,g_i(\bar x,\bar y)=0\}$.
	Setting $g_0(x,y):=f(x,y)-\varphi(x)$ for all $x\in\R^n$
	and $y\in\R^m$, we assume that there is a neighbourhood $U\subset\R^n\times\R^m$
	of $(\bar x,\bar y)$ such that for each index set $J\subset I(\bar x,\bar y)\cup\{0\}$,
	the family $(\nabla_yg_i(x,y))_{i\in J}$ has constant rank on $U$.
	Then RRCQ is valid.
\end{lemma}
\begin{proof}
	The assumptions of the lemma guarantee that for each index set $\tilde J\subset I(\bar x,\bar y)$,
	the family $(\nabla_yg_i(x,y))_{i\in\tilde J}$ possesses constant rank on $U$.
	Furthermore, $\Gamma$ is inner semicontinuous at $(\bar x,\bar y)$ by inner semicontinuity
	of $S$ at this point. We now can apply \cite[Theorem~4.2]{BednarczukMinchenkoRutkowski2019}
	in order to obtain that $\Gamma$ is R-regular at $(\bar x,\bar y)$.
	
	Observe that $\varphi$ is continuous at $\bar x$ since $S$ is inner semicontinuous at 
	$(\bar x,\bar y)$. Thus, we can invoke \cite[Theorem~5.1]{BednarczukMinchenkoRutkowski2019}
	and \cite[Theorem~5.2(i)]{MordukhovichNam2005} in order to see that $\varphi$ is already
	locally Lipschitz continuous at $\bar x$. Particularly, $\varphi$ is continuous in a
	neighbourhood of $\bar x$.

	Thus, locally around $(\bar x,\bar y)$, the variational description of $\Phi$
	is provided by functions which are Lipschitz continuous and continuously differentiable w.r.t.\
	$y$. By assumption, $\Phi$ is inner semicontinuous at $(\bar x,\bar y)$ since
	$\Phi$ and $S$ actually coincide.
	Now, the desired result follows from \cite[Theorem~4.2]{BednarczukMinchenkoRutkowski2019}
	again.
\end{proof}
\begin{remark}\label{rem:local_coincidence_of_domains}
	Observe that the assumptions of \cref{lem:RRCQ_via_CRCQ_and_lower_semicontinuity}
	imply that the point of interest $(\bar x,\bar y)\in\R^n\times\R^m$ satisfies
	$\bar x\in\intr\dom S$ since $S$ is supposed to be inner semicontinuous at
	$(\bar x,\bar y)$. Thus, we have $\bar x\in\intr\dom\Gamma$ as well.
	Consequently, the domains of $\Gamma$ and $S$ coincide locally around $\bar x$.
	Particularly, the assumptions of \cref{lem:RRCQ_via_CRCQ_and_lower_semicontinuity} 
	already guarantee that LUWSMC holds at $(\bar x,\bar y)$ due to
	\cref{lem:RRCQ_implies_LUWSMC}.
\end{remark}

In \cref{fig:partial_calmness_CQs}, we depict the relations between the conditions
from \cref{def:partial_calmness_CQs} which are all sufficient for partial calmness
at a given local minimizer of \eqref{eq:BPP}.

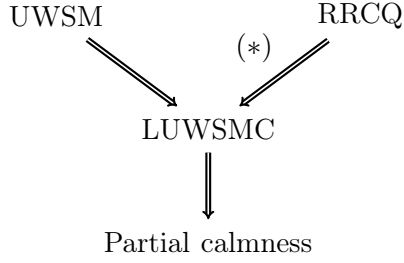
\begin{figure}[h]
\centering
\begin{tikzpicture}[->]

  \node[punkt] at (-2,0) 	(A){UWSM};
  \node[punkt] at (2,0) 	(B){RRCQ};
  \node[punkt] at (0,-1.5) 	(C){LUWSMC};
  \node[punkt] at (0,-3) 	(E){Partial calmness};

  \path     (A) edge[-implies,thick,double] node {}(C)
            (B) edge[-implies,thick,double] node[above left] {$(*)$}(C)
            (C) edge[-implies,thick,double] node {}(E);
\end{tikzpicture}
\caption{
	Relations between conditions implying partial calmness at a
	given local minimizer $(\bar x,\bar y)$ of \eqref{eq:BPP}.
	Relation $(*)$ requires local coincidence of $\dom\Gamma$ and $\dom S$ around $\bar x$.
}
\label{fig:partial_calmness_CQs}
\end{figure}

\section{Partial calmness and linear lower level problems}\label{sec:linear_lower_level}

In this section, we investigate the presence of partial calmness for the bilevel
programming problem \eqref{eq:BPP} where the lower level solution mapping
$S$ is described by one of the following settings:
\begin{itemize}
	\item the lower level problem is linear w.r.t.\ the decision variable $y$, i.e.,
		there exist continuous functions $c\colon\R^n\to\R^m$, $A\colon\R^n\to\R^q$,
		and $B\colon\R^n\to\R^{q\times m}$ such that
		\begin{equation}\label{eq:lower_level_linear_in_y}
			\forall x\in\R^n\colon\quad
			S(x):=\argmin\limits_y\{c(x)^\top y\,|\,A(x)+B(x)y\leq 0\},
		\end{equation}
	\item the lower level problem is a linear parametric optimization problem with
		continuous right-hand side perturbation, i.e., there exist matrices
		$c\in\R^m$ and $B\in\R^{q\times m}$ as well as a continuous function
		$A\colon\R^n\to\R^q$ such that
		\begin{equation}\label{eq:fully_linear_lower_level}
			\forall x\in\R^n\colon\quad
			S(x):=\argmin\limits_y\{c^\top y\,|\,A(x)+By\leq 0\},
		\end{equation}
	\item the lower level problem is a linear parametric optimization problem with
		affine perturbations of the coefficients within the objective function, i.e., there
		exist matrices $c\in\R^m$, $C\in\R^{m\times n}$, $A\in\R^q$, and $B\in\R^{q\times m}$
		such that
		\begin{equation}\label{eq:affine_perturbation_of_objective}
			\forall x\in\R^n\colon\quad
			S(x):=\argmin\limits_y\{(Cx+c)^\top y\,|\,A+By\leq 0\}.
		\end{equation}
\end{itemize}
Observe that the model \eqref{eq:lower_level_linear_in_y} covers \eqref{eq:fully_linear_lower_level}
and \eqref{eq:affine_perturbation_of_objective}.

In \cite[Proposition~4.1]{YeZhu1995}, the authors show that bilevel programming
problems with fully linear lower level problem (recall that this means that the
functions $f$ and $g$ need to be affine, and this is covered by model
\eqref{eq:fully_linear_lower_level} in case where $A$ is affine) are partially calm at all their
local minimizers. As we will show below, this result is correct although the
proof in \cite{YeZhu1995} comprises a small mistake. Subsequently, we present a slightly
more general statement than the one from \cite{YeZhu1995}, which addresses
lower level problems of type \eqref{eq:fully_linear_lower_level}, and state a corrected
version of the proof. Afterwards,
we point the reader's attention to the bug in \cite{YeZhu1995}.

\begin{theorem}\label{thm:partial_calmness_fully_linear_lower_level}
	Let $(\bar x,\bar y)\in\R^n\times\R^m$ be a local minimizer of \eqref{eq:BPP}
	where the lower level solution mapping $S$ is given as in \eqref{eq:fully_linear_lower_level}.
	Then \eqref{eq:BPP} is partially calm at $(\bar x,\bar y)$.
\end{theorem}
\begin{proof}
	Choose $\delta>0$ arbitrarily and fix some point $(x,y,u)\in\mathbb B_\delta(\bar x,\bar y,0)$
	which satisfies
	\[
		x\in X,\quad c^\top y-\varphi(x)\leq u,\quad A(x)+By\leq 0.
	\]
	Noting that this implicitly demands $S(x)\neq\varnothing$, we can pick a vector
	\[
		y(x)\in \argmin_z\{\norm{y-z}{\infty}\,|\,z\in S(x)\}
	\]
	since $S(x)$ is a polyhedron and, thus, closed.
	As a consequence, we obtain
	\begin{align*}
		&\norm{y-y(x)}{\infty}
		=\min\limits_{\sigma,z}
			\bigl\{\sigma
				\,\bigl|\,
				-\sigma\mathtt e\leq y-z\leq\sigma\mathtt e,\,c^\top z-\varphi(x)\leq 0,A(x)+Bz\leq 0
			\bigr\}
		\\
		&\quad
		=\max\limits_\xi
			\left\{
				y^\top(\xi_1-\xi_2)+\varphi(x)\xi_3-A(x)^\top\xi_4
				\,\middle|\,
					\begin{aligned}
						&\xi_1-\xi_2+c\xi_3+B^\top\xi_4=0,\\
						&-\mathtt e^\top\xi_1-\mathtt e^\top\xi_2=1,\\
						&\xi_1,\xi_2,\xi_3,\xi_4\leq 0
					\end{aligned}
			\right\}
		\\
		&\quad
		\stackrel{\textcolor{red}{(*)}}{=}
		\max\limits_\xi
			\left\{
				(\varphi(x)-c^\top y)\xi_3+(-A(x)-By)^\top\xi_4
				\,\middle|\,
					\begin{aligned}
						&\xi_1-\xi_2+c\xi_3+B^\top\xi_4=0,\\
						&-\mathtt e^\top\xi_1-\mathtt e^\top\xi_2=1,\\
						&\xi_1,\xi_2,\xi_3,\xi_4\leq 0
					\end{aligned}
			\right\}
	\end{align*}
	by strong duality of linear programming.
	Observing that the latter program possesses a solution, there exists 
	a vertex $(\xi_1(x,y),\xi_2(x,y),\xi_3(x,y),\xi_4(x,y))$ of the set
	\[
		Q
		:=
		\left\{
			(\xi_1,\xi_2,\xi_3,\xi_4)
			\,\middle|\,
					\begin{aligned}
						&\xi_1-\xi_2+c\xi_3+B^\top\xi_4=0,\\
						&-\mathtt e^\top\xi_1-\mathtt e^\top\xi_2=1,\\
						&\xi_1,\xi_2,\xi_3,\xi_4\leq 0
					\end{aligned}
		\right\}
	\]
	which satisfies
	\begin{align*}
		\norm{y-y(x)}{\infty}
		&=
		(\varphi(x)-c^\top y)\xi_3(x,y)+(-A(x)-By)^\top\xi_4(x,y)\\
		&\leq
		(\varphi(x)-c^\top y)\xi_3(x,y)
		\leq
		|u|\,|\xi_3(x,y)|.
	\end{align*}
	Noting that the polyhedron $Q$ possesses only finitely many vertices and
	does not depend on the choice of $(x,y)$, there is some constant $M>0$
	such that $\norm{y-y(x)}{\infty}\leq M|u|$ follows.
	Observing that $F$ is locally Lipschitz continuous, we find some
	constant $L>0$ such that $F$ is Lipschitz continuous with Lipschitz
	constant $L$ on the ball $\mathbb B_{(M+1)\delta}(\bar x,\bar y)$
	if only $\delta$ is small enough.
	Since we have
	\[
		\norm{y(x)-\bar y}{\infty}
		\leq
		\norm{y(x)-y}{\infty}+\norm{y-\bar y}{\infty}
		\leq
		M|u|+\norm{y-\bar y}{\infty}
		\leq
		(M+1)\delta
	\]
	and $\norm{x-\bar x}{\infty}\leq\delta$,
	we can choose $\delta$ so small such that $(x,y(x))$ always lies within the
	radius of local optimality associated with the local minimizer $(\bar x,\bar y)$
	of \eqref{eq:BPP}. Noting that $(x,y(x))$ is feasible to \eqref{eq:BPP}, it holds
	\[
		F(x,y)-F(\bar x,\bar y)
		\geq
		F(x,y)-F(x,y(x))
		\geq
		-L\,\norm{y-y(x)}{\infty}
		\geq
		-LM|u|.
	\]
	Recalling that $(x,y,u)\in\mathbb B_\delta(\bar x,\bar y,0)$ has been chosen
	arbitrarily, the statement of the theorem follows.
\end{proof}

Let us recall that in \cite[Section~4.2]{YeZhu1995}, the authors consider the
particular case where the function $A$ in the definition of \eqref{eq:fully_linear_lower_level}
is affine.

Next, we specify where the bug in the original proof from \cite{YeZhu1995}
is located. If not stated otherwise, the subsequently stated remarks address
the lower level problem from \eqref{eq:fully_linear_lower_level}.
\begin{remark}\label{rem:the_classical_bug}
	In the classical proof of \cite[Proposition~4.1]{YeZhu1995},
	it has been claimed that equality in $\textcolor{red}{(*)}$ holds
	with the right-hand side
	\begin{equation}\label{eq:problem_wrong_estimate}
		\max\limits_\xi
			\left\{
				(\varphi(x)-c^\top y)\xi_3+(-A(x)-By)^\top\xi_4
				\,\middle|\,
					\begin{aligned}
						&-\mathtt e^\top\xi_1-\mathtt e^\top\xi_2=1,\\
						&\xi_1,\xi_2,\xi_3,\xi_4\leq 0
					\end{aligned}
			\right\}
	\end{equation}
	where the constraint $\xi_1-\xi_2+c\xi_3+B^\top\xi_4=0$ is deleted from
	the feasible set. Noting that this enlarges the feasible set of the
	program from the left-hand side of $\textcolor{red}{(*)}$, the equality there
	needs to be replaced by the relation $\leq$.
	Even worse, it is obvious that whenever $y\notin S(x)$
	holds, then, due to $\varphi(x)-c^\top y<0$,
	\eqref{eq:problem_wrong_estimate} possesses the optimal
	value $+\infty$. Thus, the authors obtained the trivial estimate
	$\norm{y-y(x)}{\infty}\leq+\infty$ in \cite{YeZhu1995} which is, for sure,
	of no use.
	
	Clearly, the proof provided above cannot be generalized to the setting where $S$
	is given as in \eqref{eq:lower_level_linear_in_y}, since in this case, the vertices of the set
	\[
		Q(x)
		:=
		\left\{
			(\xi_1,\xi_2,\xi_3,\xi_4)
			\,\middle|\,
					\begin{aligned}
						&\xi_1-\xi_2+c(x)\xi_3+B(x)^\top\xi_4=0,\\
						&-\mathtt e^\top\xi_1-\mathtt e^\top\xi_2=1,\\
						&\xi_1,\xi_2,\xi_3,\xi_4\leq 0
					\end{aligned}
		\right\}
	\]
	depend on the parameter $x$ which means that the existence of the constant
	$M$ in the proof of \cref{thm:partial_calmness_fully_linear_lower_level}
	does not come for free.
	These arguments also demonstrate that the proof of
	\cite[Theorem~4.2]{DempeZemkoho2013} is not correct since it
	reprises the original bug from \cite{YeZhu1995}.
\end{remark}

\begin{remark}\label{rem:uniformly_weak_sharp_minimum}
	Inspecting the proof of \cref{thm:partial_calmness_fully_linear_lower_level}
	carefully, one can observe that the condition
	\begin{align*}
		&\exists M>0 \;\; \forall (x,y)\in\R^n\times\R^m\colon\\
		&\qquad A(x)+By\leq 0
		\,\Longrightarrow\,
		\dist(y,S(x))\leq M(c^\top y-\varphi(x))
	\end{align*}
	has been verified, and the latter already means that the parametric
	lower level optimization problem associated with \eqref{eq:fully_linear_lower_level}  
	possesses a UWSM,
	see \cref{def:partial_calmness_CQs}.
	Recently, this has been pointed out in \cite[Lemma~2.1]{MinchenkoBerezhov2017}.
\end{remark}

\begin{remark}\label{rem:fully_linear_lower_level_and_calmness}
	Consider the situation where the mapping $A$ is affine.
	In this case, it is well known that
	the solution map $S$ associated with \eqref{eq:fully_linear_lower_level} is inner
	semicontinuous at each point $(\bar x,\bar y)\in\gph S$ which satisfies
	$\bar x\in\intr\dom S$.
	Thus, \cref{lem:RRCQ_via_CRCQ_and_lower_semicontinuity} guarantees that 
	at any such point, RRCQ is valid since
	the remaining constant rank assumption trivially holds observing that all appearing gradients
	w.r.t.\ $y$ are constant.
\end{remark}

In the light of \cref{rem:the_classical_bug}, one now might ask whether the result of \cref{thm:partial_calmness_fully_linear_lower_level}
can be generalized to the setting where the
lower level problem is given as in \eqref{eq:lower_level_linear_in_y} as proposed in
\cite[Theorem~4.2]{DempeZemkoho2013}.
As the following example shows, this is, unluckily, not the case.
\begin{example}\label{ex:linear_lower_level_no_partial_calmness}
	Let us consider the bilevel optimization problem
	\[
		\min\limits_{x,y}\{-x+y\,|\,x\leq 2,\,y\in S(x)\}
	\]
	where $S\colon\R\rightrightarrows\R$ is given by
	\[
		\forall x\in\R\colon\quad
		S(x):=\argmin\limits_y\{-x^2y\,|\,y\in[0,1]\}.
	\]
	One can easily check that
	\[
		\forall x\in\R\colon\quad
		S(x)
		:=
		\begin{cases}
			\{1\}	&\text{if } x\neq 0,\\
			[0,1]	&\text{if } x=0
		\end{cases}
		\qquad
		\mbox{and}
		\qquad
		\varphi(x)
		:=
		-x^2
	\]
	hold true, i.e., $(\bar x,\bar y):=(0,0)$ is a local minimizer of the given bilevel programming problem. Furthermore, the global minimizer of this bilevel optimization problem
	is given by $(\tilde x,\tilde y):=(2,1)$.
	
	Now, for $\kappa>0$, let us consider the associated partially penalized problem
	\eqref{eq:partially_penalized_OVR}
	which reads as
	\begin{equation}\label{eq:partially_penalized_LLVF_example}
		\min\limits_{x,y}\{-x+y+\kappa x^2(1-y)\,|\,x\leq 2,\,y\in[0,1]\}.
	\end{equation}
	The sequence $\{(\tfrac1k,0)\}_{k\in\N}$ is feasible for the latter and converges to $(\bar x,\bar y)$.
	The associated objective values of \eqref{eq:partially_penalized_LLVF_example}
	are given by $\{-\tfrac1k+\kappa\tfrac1{k^2}\}_{k\in\N}$. Observe that for sufficiently large $k\in\N$,
	the elements of this sequence are negative. Particularly, there is no finite $\kappa>0$ such that
	$(\bar x,\bar y)$ is a local minimizer of \eqref{eq:partially_penalized_LLVF_example}.
	Due to \cref{lem:partial_calmness_and_exact_penalization},
	the bilevel optimization problem under consideration cannot
	be partially calm at $(\bar x,\bar y)$.
\end{example}

The above example refutes \cite[Theorem~4.2]{DempeZemkoho2013}.
The proof provided in the latter paper essentially adapted
the one from \cite{YeZhu1995} and the bug therein, see \cref{rem:the_classical_bug}.
In contrast to \cite{YeZhu1995}, where this mistake did
not effect the correctness of the result, the statement from \cite{DempeZemkoho2013} is not true in general.

In the light of \cref{lem:LUWSM_ensures_VFCQ}, the following result, however, provides a sufficient condition for partial
calmness in cases where the lower level problem is given as in \eqref{eq:lower_level_linear_in_y}.
It follows directly from \cref{lem:RRCQ_via_CRCQ_and_lower_semicontinuity},
see \cref{rem:local_coincidence_of_domains} as well.
\begin{theorem}\label{thm:linear_only_in_y}
	Let us assume that the solution map $S\colon\R^n\rightrightarrows\R^m$ is given
	as in \eqref{eq:lower_level_linear_in_y} where the matrix functions $c$, $A$, and
	$B$ are assumed to be locally Lipschitz continuous.
	Fix a point $(\bar x,\bar y)\in\gph S$ where $S$ is inner semicontinuous.
	Furthermore, set
	\[
		\forall x\in\R^n\colon\quad
		\mathcal B(x):=
		\begin{bmatrix}
			c(x)^\top\\
			B(x)_{I(\bar x,\bar y)}
		\end{bmatrix}
	\]
	and assume that there is a neighbourhood $U\subset \R^n$ of $\bar x$
	such that for each index set  $J\subset\{1,\ldots,|I(\bar x,\bar y)|+1\}$,
	the matrix $\mathcal B(x)_J$ has constant row rank for all $x\in U$.
	Then RRCQ holds at $(\bar x,\bar y)$.
\end{theorem}

Observe that due to \cref{rem:fully_linear_lower_level_and_calmness}, \cref{thm:linear_only_in_y}
covers \cref{thm:partial_calmness_fully_linear_lower_level} in the setting where
$A$ is an affine function while $c$ and $B$ are constant.
Clearly, the assumptions of \cref{thm:linear_only_in_y} cannot be neglected when
considering partial calmness of \eqref{eq:BPP} -
checking \cref{ex:linear_lower_level_no_partial_calmness}, one obtains that
both the inner semicontinuity assumption on $S$ and the constant rank assumption
are violated at the point of interest.
Below, we present an example where all these assumptions hold.
\begin{example}\label{ex:RRCQ_via_CRCQ}
	We consider the lower level solution map $S\colon\R\rightrightarrows\R$ given
	in \cref{ex:linear_lower_level_no_partial_calmness} at 
	$(\tilde x,\tilde y):=(2,1)$. 
	Clearly, $S$ is inner semicontinuous at $(\tilde x,\tilde y)$ and
	the matrix
	\[
		\forall x\in\R\colon\quad
		\mathcal B(x)=\begin{bmatrix}-x^2\\1\end{bmatrix}
	\]
	satisfies the constant rank assumption from \cref{thm:linear_only_in_y}
	in a neighbourhood of $\tilde x$.
	Thus, RRCQ is valid at $(\tilde x,\tilde y)$.
	In the light of \cref{rem:local_coincidence_of_domains} and
	\cref{lem:LUWSM_ensures_VFCQ}, the bilevel optimization problem
	from \cref{ex:linear_lower_level_no_partial_calmness} is
	partially calm at its global minimizer $(\tilde x,\tilde y)$.
\end{example}

The next example illustrates that even in the presence of inner semicontinuity
of the solution mapping $S$ at the point of interest, partial calmness does not
need to be inherent for problems with the lower level problem \eqref{eq:lower_level_linear_in_y}
if the constant rank assumption from \cref{thm:linear_only_in_y} is violated.
\begin{example}\label{ex:lower_level_with_lsc_solution_map}
	Let us investigate the bilevel optimization problem
	\[
		\min\limits_{x,y}\{xy_1\,|\,y\in S(x)\}
	\]
	where $S\colon\R\rightrightarrows\R^2$ is given by
	\[
		\forall x\in\R\colon\quad
		S(x):=\argmin\limits_y\{-x^2y_2\,|\,y_2\leq 0,\,-xy_1+y_2\leq 0\}.
	\]
	One can easily check that the corresponding solution mapping $S$ and 
	the optimal value function $\varphi$, respectively,
	take the following precise forms: 
	\[
		\forall x\in\R\colon\quad 
		S(x)=
			\begin{cases}
				(-\infty,0]\times\{0\}	&x<0,\\
				\R\times(-\infty,0]		&x=0,\\
				[0,\infty)\times\{0\}	&x>0
			\end{cases}
		\qquad\mbox{and}\qquad
		\varphi(x)=0.
	\]
	Furthermore, one can check that $(\bar x,\bar y):=(0,(0,0))$ 
	is a global minimizer of the given bilevel
	optimization problem. Clearly, $S$ is inner semicontinuous at this point.
	The associated matrix $\mathcal B$ given by
	\[
		\forall x\in\R\colon\quad
		\mathcal B(x)=
		\begin{bmatrix}
			0	&	-x^2\\	0	&	1\\ -x	&	1
		\end{bmatrix}
	\]
	does not satisfy the constant rank assumption at $\bar x$.\\
	For $\kappa>0$, we consider the associated partially penalized problem \eqref{eq:partially_penalized_OVR}
	given by
	\[
		\min\limits_{x,y}\{xy_1-\kappa x^2y_2\,|\,y_2\leq 0,\,-xy_1+y_2\leq 0\}.
	\]
	Investigating the feasible sequence $\{(\tfrac{1}{k},(-\tfrac{1}{k},-\tfrac{1}{k^2}))\}_{k\in\N}$,
	the associated objective values are $\{\tfrac{\kappa}{k^4}-\tfrac1{k^2}\}_{k\in\N}$, and this shows
	that $(\bar x,\bar y)$ is not a local minimizer of the latter problem for any $\kappa>0$ since
	the elements of the latter sequence are negative for sufficiently large $k\in\N$.
	Hence, the underlying bilevel optimization problem is not partially calm at $(\bar x,\bar y)$.
\end{example}

Let us briefly mention that \cite[Theorem~3.4]{Ye1998} provides a condition
which ensures that the lower level problem \eqref{eq:lower_level_linear_in_y}
even possesses a UWSM, see \cref{def:partial_calmness_CQs}.
On the other hand, in \cite[Example~3.9]{HenrionSurowiec2011}, it has been shown that
already parametric optimization problems of type \eqref{eq:affine_perturbation_of_objective}
do not necessarily satisfy LUWSMC and, thus, cannot possess a UWSM, i.e., the assumptions of \cite[Theorem~3.4]{Ye1998} are
not generally satisfied for this class of lower level problems.
In the light of \cref{rem:uniformly_weak_sharp_minimum}, this observation already underlines that trying
to adapt the proof of \cref{thm:partial_calmness_fully_linear_lower_level} is hopeless in order to infer
the partial calmness condition for bilevel programming problems with lower level problem
\eqref{eq:affine_perturbation_of_objective} and, thus, \eqref{eq:lower_level_linear_in_y}.
Furthermore, we would like to note that the solution mapping from \eqref{eq:affine_perturbation_of_objective}
is not generally inner semicontinuous at the points of its graph which restricts the
applicability of \cref{thm:linear_only_in_y}.

We want to close this section with an example which illustrates that bilevel optimization
problems with lower level problems of type \eqref{eq:affine_perturbation_of_objective}
indeed do not need to be partially calm at their respective local minimizers.
This underlines that standard models from bilevel road pricing as discussed in
\cite{DempeFranke2015,DempeZemkoho2012} are not generally partially calm at their
local minimizers without additional assumptions.
\begin{example}\label{ex:no_partial_calmness_for_affinely_perturbed_lower_level_objective}
	We investigate the bilevel optimization problem
	\[
		\min\limits_{x,y}\{x_2(y+1)\,|\,x_1=x_2^2,\,y\in S(x)\}
	\]
	where $S\colon\R^2\rightrightarrows\R$ is given by
	\[
		\forall x\in\R^2\colon\quad
		S(x):=\argmin\limits_y\{x_1y\,|\,y\in[-1,1]\}.
	\]
	One obtains
	\[
		\forall x\in\R^2\colon\quad
		S(x)
		=
		\begin{cases}
			\{-1\}		&x_1>0,\\
			[-1,1]		&x_1=0,\\
			\{1\}		&x_1<0
		\end{cases}
		\qquad
		\mbox{and}
		\qquad
		\varphi(x)
		=
		-|x_1|
	\]
	by simple calculations. One can easily check that each feasible point of this
	bilevel optimization problem possesses objective value $0$.
	Particularly, $(\bar x,\bar y):=((0,0),-1)$ is one of its global minimizers.
	Next, for arbitrary $\kappa>0$, we consider the associated partially
	perturbed problem \eqref{eq:partially_penalized_OVR} given by
	\begin{equation}\label{eq:partially_perturbed_OVR_example_2}
		\min\limits_{x,y}\{x_2(y+1)+\kappa(x_1y+|x_1|)\,|\,x_1=x_2^2,\,y\in[-1,1]\}.
	\end{equation}
	We investigate the feasible sequence $\{((\tfrac{1}{k^2},-\tfrac{1}{k}),-1+\tfrac1k)\}_{k\in\N}$
	which converges to $(\bar x,\bar y)$.
	The associated sequence of objective values is given by $\{\tfrac{\kappa}{k^3}-\tfrac{1}{k^2}\}_{k\in\N}$,
	and the latter is negative for sufficiently large $k\in\N$.
	This shows that $(\bar x,\bar y)$ is not a local minimizer of
	\eqref{eq:partially_perturbed_OVR_example_2}, and due to \cref{lem:partial_calmness_and_exact_penalization},
	the bilevel programming problem under consideration cannot be partially calm at
	$(\bar x,\bar y)$.
\end{example}

We would like to point the reader's attention to the fact that the upper level feasible set $X$ in
\cref{ex:no_partial_calmness_for_affinely_perturbed_lower_level_objective}
is not a polyhedron, which is the striking idea behind the construction
of this counterexample. It is, however, an open question whether or not
bilevel optimization problems with lower level problem \eqref{eq:affine_perturbation_of_objective}
and polyhedral $X$ are partially calm at their respective local minimizers.
Observing that the solution mapping $S$ is a polyhedral set-valued mapping in this
case (i.e., its graph can be represented as the union of finitely many convex
polyhedral sets), this indeed might be possible since the associated feasible set of the bilevel
optimization problem \eqref{eq:BPP} is the union of finitely many convex polyhedral
sets, and problems of this type are likely to be calm in Clarke's sense
due to \cite[Proposition~1]{Robinson1981}.

\section{Conclusions}\label{sec:conclusion}

This manuscript underlines the well known fact that the property of a bilevel optimization problem
to be partially calm at one of its local minimizers is quite restrictive.
With the aid of three simple counterexamples, we have shown that this observation
already addresses situations where the lower level problem is linear w.r.t.\ the lower level
decision maker's variable.
Our respective analysis refutes the result \cite[Theorem~4.2]{DempeZemkoho2013}.
On the way, we revealed and corrected a bug in the proof of the seminal result \cite[Proposition~4.1]{YeZhu1995}
which has been spread over the literature about bilevel optimization.


%

\end{document}